\documentclass{article}

\usepackage{amsmath,
            amssymb,
            latexsym}

\usepackage[pdftex]{graphicx} 
\usepackage{epstopdf} 

\newtheorem{theorem}{Theorem}[section]
\newtheorem{lemma}[theorem]{Lemma}
\newtheorem{corollary}[theorem]{Corollary}

\newtheorem{proposition}[theorem]{Proposition}
\newtheorem{definition}[theorem]{Definition}

\newenvironment{proof}{\noindent\textsc{Proof: }}
{\hspace{\stretch{1}}$\Box$\medskip}

\begin{document}

\title{Cohomological Ramsey Theory}

\author{
  Alexander Engstr\"om\footnote{The author is Miller Research Fellow 2009-2012 at UC Berkeley, and gratefully acknowledges support from the Adolph C. and Mary Sprague Miller Institute for Basic Research in Science.}
  \\
Department of Mathematics \\
UC Berkeley \\ 
\texttt{alex@math.berkeley.edu}
}

\date\today

\maketitle

\begin{abstract}
We show that the vanishing of certain cohomology groups of polyhedral complexes imply upper bounds on Ramsey numbers. Lov\'asz bounded the chromatic numbers of graphs
using Hom complexes. Babson and Kozlov proved Lov\'asz conjecture and developed a Hom complex theory. We generalize the Hom complexes to Ramsey complexes.

The main theorem states that if certain cohomology groups of the Ramsey complex ${\tt Ram}( \partial \Delta_{p^k}, \Sigma)$ are trivial, then the vertices of the simplicial complex $\Sigma$ cannot be $n$-colored such that every color correspond to a face of $\Sigma$. In a corollary, we give an explicit description of the Ramsey complexes used for upper bounds on Ramsey numbers.
\end{abstract}

\section{Introduction}

In this paper we show that the cohomology of polyhedral complexes can be used for upper bounds on Ramsey numbers.
Lov\'asz proved that the chromatic number of a graph can be bounded by the connectivity of polyhedral complexes \cite{L}. 
Later Babson and Kozlov \cite{BK} proved a conjecture by Lov\'asz that extended the complexes used to find chromatic numbers to certain Hom complexes. The neighborhood complexes used by Lov\'asz originally are special cases of the Hom complexes.

We define a generalization of the Hom complexes, called Ramsey complexes. A Ramsey complex ${\tt Ram}(\Sigma_1,\Sigma_2)$ is defined using two simplicial complexes while the Hom complex ${\tt Hom}(G_1,G_2)$ used two graphs. The generalization is \emph{not} that two-dimensional simplicial complexes corresponds to graphs, but that the simplicial complexes correspond to independence complexes of graphs. This is our main theorem:
\newline
\newline
\noindent  {\bf Theorem \ref{theTheorem}}\emph{
Let $\Sigma$ be a simplicial complex and $p^k$ a prime power.
If  
\[ \tilde{H}^i(  Ram( \partial \Delta_{p^k}, \Sigma)  ;\mathbb{Z}_p)=0\]
 for all $i \leq (n-1)(p^k-1)-1$ then
the vertices of $\Sigma$ cannot be $n$-colored such that each color is a face of $\Sigma$. 
}
\newline
\newline
Setting $p^k=2$ and $\Sigma$ to the independence complex of a graph, one recovers Lov\'asz result.

The Ramsey number $R(G;n)$ is the smallest number $N$ such that any edge-coloring of $K_N$ with $n$ colors will have a one-colored copy of $G$. To prove upper bounds for Ramsey numbers $R(G;n)$ we use a special type of Ramsey complexes called Rainbow complexes. They have an explicit description.
\newline
\newline
\noindent  {\bf Definition \ref{def:rain}}\emph{
The \emph{Rainbow complex to prove that $R(G;n)\leq N$ using $m$,}
is a polyhedral subcomplex of the product
\[ \prod_{i=1}^m \Delta_{E(K_N)}. \]
The vertices of the simplex $\Delta_{E(K_N)}$ are indexed by the edges of $K_N$. Any vertex of the product
corresponds to a function
\[ \eta : \{1,2,\ldots, m\} \rightarrow E(K_N). \]
The Rainbow complex is the induced subcomplex on the vertices $\eta$ such that $G$ is a subgraph of 
$\eta( \{1,2, \ldots,  m\} ).$
}
\newline
\newline
\noindent  {\bf Corollary \ref{cor:rain}}\emph{
Let $X$ be the Rainbow complex to prove that $R(G;n)\leq N$ using $m$. If $m=p^k$ is a prime power
and $\tilde{H}^i( X   ;\mathbb{Z}_p)=0$ for all $i \leq (n-1)(p^k-1)-1$, then  $R(G;n)\leq N$.
}

\section{The Ramsey complex}

For basic notions of topological combinatorics not defined in the text, we refer to Kozlov's book \cite{K}.

Most of the polyhedral complexes in this paper are subcomplexes of products of simplices. A simplex with vertex set $V$ is denoted $\Delta_V$, and
a product of them indexed by the set $I$ is
\[ X=\prod_{i \in I} \Delta_V. \]
The vertices of $X$ are indexed by functions $\eta: I \rightarrow V$, and the cells are indexed by functions $\nu: I \rightarrow 2^V \setminus \emptyset$.
By abuse of notation we write $\eta\in X$ and $\nu\in X$. A cell $\nu\in X$ is in the \emph{induced subcomplex of $X$ with vertex set $E$} if 
\[ \eta(i)\in \nu(i)\textrm{ for all }i\in I \]
implies that $\eta\in E$.

The main player of the paper is the Ramsey complex. It is a new type of polyhederal complex that generalize older constructions.
\begin{definition}
Let $\Sigma_1$ and $\Sigma_2$ be two finite simplicial complexes. The \emph{Ramsey complex} ${\tt Ram}(\Sigma_1,\Sigma_2)$ is
the induced subcomplex of
\[ \prod_{v \in \Sigma_1^0} \Delta_{\Sigma_2^0} \]
with vertex set
\[  \{ \eta:  \Sigma_1^0 \rightarrow \Sigma_2^0 \mid \eta^{-1}(\sigma)\in \Sigma_1\textrm{ for all }\sigma\in \Sigma_2 \}. \]
\end{definition}

A \emph{graph homomorphism} $G_1\rightarrow G_2$ is a map $\eta: V(G_1)\rightarrow V(G_2)$ such that $\eta(u)\eta(v)$ is an edge of $G_2$ if
$uv$ is an edge of $G_1$.
\begin{definition} Let $G_1$ and $G_2$ be graphs. The \emph{Hom complex} ${\tt Hom}(G_1,G_2)$ is the induced subcomplex of
\[  \prod_{v \in V(G_1)} \Delta_{V(G_2)} \]
with vertex set
\[  \{ \eta: V(G_1)\rightarrow V(G_2) \mid \eta \textrm{ is a graph homomorphism }G_1 \rightarrow G_2 \}. \]
\end{definition}

The Hom complex was introduced by Lov\'asz to generalize the neighborhood complex he used in the proof of Kneser's conjecture \cite{L}.
Babson and Kozlov created a theory of Hom complexes and proved Lov\'asz conjecture \cite{BK}. \v{C}uki\'c and Kozlov \cite{Cu} explained the Hom complexes with
complete graphs in the second component.
The strength of different topological tests for graphs was studied by
Csorba \cite{C}, and combinatorial versions of Lov\'asz proof was studied by Ziegler \cite{Z}. The homotopy theoretic perspective was investigated by Dochtermann \cite{D}.
Extending an idea by Zivaljevic \cite{Z}, Schultz found new proofs and explained more how topological tests works. In the last part of Kozlov's book \cite{K} on
combinatorial algebraic topology most of this is surveyed.

\begin{definition} Let $\Sigma$ be a simplicial complex and $V$ a set. The \emph{Partition complex} ${\tt Part}(\Sigma,V)$ is the induced subcomplex of
\[  \prod_{v \in \Sigma^0} \Delta_{V} \]
with vertex set
\[  \{ \eta: \Sigma^0 \rightarrow V \mid \eta^{-1}(v)\in \Sigma\textrm{ for all }v\in V \} \]
\end{definition}

A set of vertices $I$ of the graph $G$ is \emph{independent} if no vertices in $I$ are adjacent.
\begin{definition} Let $G$ be a graph. The \emph{independence complex} ${\tt Ind}(G)$ is a simplicial complex with the same vertex set as $G$ and its
independent sets as faces.
\end{definition}

\begin{proposition} If $G_1$ and $G_2$ are graphs then \[ {\tt Hom}(G_1,G_2)={\tt Ram}({\tt Ind}(G_1),{\tt Ind}(G_2)).\]
\end{proposition}
\begin{proof}
The function $\eta$ is a vertex of ${\tt Ram}({\tt Ind}(G_1),{\tt Ind}(G_2))$ if
\[ \eta^{-1}(\sigma)\in {\tt Ind}(G_1)\textrm{ for all }\sigma\in {\tt Ind}(G_2). \]
That the inverse of independent sets are independent is the same as that $\eta$ is a graph homomorphism.
The $\eta \in  {\tt Hom}(G_1,G_2)$ are the graph homomorphisms.

\end{proof}

The simplicial complex of $n$ disjoint vertices $1,2,\ldots, n$ is denoted $D_n$. The $(n-1)$--dimensional simplex on $\{1,2,\ldots, n\}$ is denoted $\Delta_n$.
\begin{proposition} If $\Sigma$ is a simplicial complex, then ${\tt Part}(\Sigma,\{1,2,\ldots,n\})={\tt Ram}(\Sigma,D_n)$.
\end{proposition}
\begin{proof}
Compare the definitions.
\end{proof}

One particular instance of the two previous partitions is that ${\tt Hom}(G,K_n)={\tt Part}({\tt Ind}(G),\{1,2,\ldots,n\})={\tt Ram}({\tt Ind}(G),D_n)$. In \cite{E} most results on\linebreak
${\tt Hom}(G,K_n)$ were extended to ${\tt Part}(\Sigma,\{1,2,\ldots,n\})$ with more transparent\linebreak proofs, but without any condition on $\Sigma$. 

\begin{proposition} \label{prop:comp}
If the vertices \mbox{$\eta_1 \in {\tt Ram}(\Sigma_1,\Sigma_2)$} and \mbox{$\eta_2 \in {\tt Ram}(\Sigma_2,\Sigma_3)$} then\linebreak $\eta_2\eta_1 \in {\tt Ram}(\Sigma_1,\Sigma_3)$.
\end{proposition}
\begin{proof}
For any $\sigma\in \Sigma_3,$
\[ \eta_2^{-1}(\sigma)\in \Sigma_2\textrm{ since }\eta_2 \in {\tt Ram}(\Sigma_2,\Sigma_3),\]
\[ (\eta_2\eta_1)^{-1}(\sigma)= \eta_1^{-1}( \eta_2^{-1}(\sigma))\in \Sigma_1\textrm{ since }\eta_1 \in {\tt Ram}(\Sigma_1,\Sigma_2).\]
\end{proof}

Now we define group actions on the Ramsey complexes. As for Hom complexes, we could do this in both the first and the second component. We could also give a category theory interpretation of Proposition~\ref{prop:comp} as done by Schultz for Hom complexes in Proposition 2.9 of \cite{S}. But we refrain from doing so, and keep to what's
needed in this paper.

For any group $\Gamma \subseteq \mathcal{S}_{\Sigma_1^0}$ that extends to a simplicial action on $\Sigma_1$, there is an induced $\Gamma$--action on ${\tt Ram}(\Sigma_1,\Sigma_2)$ by
\[ (\gamma,\eta) \rightarrow \eta\gamma \textrm{ for }\gamma\in \Gamma\textrm{ and }\eta\in {\tt Ram}(\Sigma_1,\Sigma_2)^0.\]

\begin{proposition} \label{prop:doubleEqui}
If ${\tt Ram}(\Sigma_2,\Sigma_3)$ is non-empty, and $\Gamma \subseteq \mathcal{S}_{\Sigma_1^0}$ extend to a simplicial action on $\Sigma_1$, then
there is a $\Gamma$--equivariant map
\[ f : {\tt Ram}(\Sigma_1,\Sigma_2) \rightarrow {\tt Ram}(\Sigma_1,\Sigma_3) \]
\end{proposition}
\begin{proof}
Fix an element $\eta_2 \in {\tt Ram}(\Sigma_2,\Sigma_3)$. Define, using Proposition~\ref{prop:comp}, the map $f$ as $f(\eta_1)=\eta_2\eta_1$. The $\Gamma$--actions on
${\tt Ram}(\Sigma_1,\Sigma_2)$ and ${\tt Ram}(\Sigma_1,\Sigma_3)$ are induced from the action on their vertices. For any $\gamma \in \Gamma,$
\[ (\gamma,f(\eta_1))=f(\eta_1)\gamma=\eta_2\eta_1\gamma=f(\eta_1\gamma)=f((\gamma,\eta_1)).\]
\end{proof}

In the theory of Hom complexes \cite{BK} the concept of a test graphs is important. Lov\'asz used an edge as test graph, and Babson and Kozlov proved Lov\'asz conjecture that
also odd cycles can be used \cite{BK}. We use simplicial complexes instead of graphs, and the complex $\Sigma_1$ in the map of Proposition~\ref{prop:doubleEqui} is the
possible test complex. The independence complex of Lov\'asz test graph $K_2$ is $\partial \Delta_2$, and in the Ramsey complexes we use the generalization $\Sigma_1=
\partial \Delta_m$.

\begin{proposition}\label{prop:fpf}
If $\Gamma  \subseteq \mathcal{S}_{\partial \Delta_m^0}$ acts transitively on $\{1,2,\ldots m\}$, then it acts without fixed-points on ${\tt Ram}(\partial \Delta_m, \Sigma)$ for all
non-empty $\Sigma$.
\end{proposition}
\begin{proof}
Say that there is a fixed point. Then this point is included in a cell $\nu \in {\tt Ram}(\partial \Delta_m, \Sigma)$ such that $\{ \eta\gamma \mid \eta\in \nu \} = \{ \eta\in \nu \}$
for every $\gamma \in \Gamma$. We could assume this since the Ramsey complexes are induced from their one-skeletons. 
Let $\eta'$ be a vertex of $\nu$ and $v$ a vertex of $\Sigma$. By the transitivity of $\Gamma$,
\[ E=\{\eta'\gamma \mid \gamma\in \Gamma \} \subseteq \{ \eta \in \nu \} \]
and there is a $\eta \in E$ with $\eta(u)=v$ for every $u\in \partial \Delta_m^0.$ The vertex $\tilde{\eta}$ defined by $\tilde{\eta}(u)=v$ for all $v$ is also in the cell $\nu$ by
the definition of the Ramsey complex. But $\tilde{\eta}^{-1}(v)=\partial \Delta_m^0$ and $\tilde{\eta}\not\in \partial \Delta_m$, which is a contradiction.
\end{proof}

The previous proposition could also have been proved using the geometrical realization we are about to define now.

Now we define the geometrical realization of ${\tt Ram}(\partial \Delta_m, D_n)$ that is used in the rest of the paper.
The complex ${\tt Ram}(\partial \Delta_m, D_n)$ is a subcomplex of
\[ \prod_{i=1}^m \Delta_n =  \left\{ (x_{ij}) \in \mathbb{R}^{m(n-1)} \left| \begin{array}{l} 
x_{ij} \geq 0\textrm{ for all }1\leq i\leq m \textrm{ and } 1\leq j\leq n-1 
\\   x_{i1}+x_{i2}+\cdots+x_{i(n-1)}=1\textrm{ for all } 1\leq i \leq m
 \end{array} \right. \right\}. \]
 The vertex $\eta$ is realized as
 \[ x_{ij} = \left\{\begin{array}{cl}
 1 & \eta(i)=j, \\
 0 & \eta(i)\neq j. 
  \end{array} \right. \]
The smallest cell $\nu$ containing a point $x$ is defined by
\[ \nu(i)= \left\{ 1\leq j \leq n \,\, \left| \begin{array}{l}
\textrm{if $j<n$ then $x_{ij}>0$} \\
\textrm{if $j=n$ then $x_{i1}+x_{i2}+\cdots+x_{i(n-1)}<1$}
\end{array}\right.
\right\}.\]
An action of  $\Gamma\subseteq\mathcal{S}_{\partial \Delta_m}=\mathcal{S}_m$ on ${\tt Ram}(\partial \Delta_m, D_n)$ is realized component-wise as
\[ \gamma(x_{ij})=x_{\gamma(i)j}\]
for $\gamma \in \Gamma$.

In equivariant obstruction theory one usually maps a complex into a space minus a diagonal. The obstruction occurs when something lands on the diagonal.
The \emph{diagonal} $\mathbb{D}^{m,n-1}$ is
\[ \{ (x_{ij}) \in \mathbb{R}^{m(n-1)}  \mid x_{1j}=x_{2j}= \cdots =x_{mj}\textrm{ for all }1\leq j\leq n-1 \}. \]

\begin{lemma}\label{lemma:intoDiag}
If $\Gamma \subseteq S_m$,  then $\Gamma$ acts on ${\tt Ram}(\Delta_m,T_n)$, and on $\mathbb{R}^{m(n-1)} \setminus \mathbb{D}^{m,n-1}$; and 
the there is an equivariant inclusion map
\[ {\tt Ram}(\partial \Delta_m,D_n) \rightarrow \mathbb{R}^{m(n-1)} \setminus \mathbb{D}^{m,n-1}. \]
\end{lemma}  
\begin{proof}
If $\gamma$ is a permutation on $S_m$ then it acts on $\mathbb{R}^{m(n-1)}$ by
\[ x_{ij} \rightarrow x_{\gamma(i)j}. \]
Now assume that $x$ is in the intersection of  $\mathbb{D}^{m,n-1}$ and ${\tt Ram}(\partial \Delta_m,D_n)$.
If any $x_{ij}$ would be negative or larger than one, then $x$ would be outside ${\tt Ram}(\partial \Delta_m,D_n)$ since it is not even in $\prod_{i=1}^m \Delta_n$.
\begin{itemize}
\item[1.] If $x_{k1}+x_{k2}+\ldots+x_{k(n-1)}=0$ for all $1\leq k\leq m$, then $x$ is the vertex
\[ \eta: \{1,2,\ldots, m\} \rightarrow \{1,2,\ldots, n\} \textrm{ defined by } \eta(i)=j \]
where $j=n$.
\item[2.] If $x_{k1}+x_{k2}+\ldots+x_{k(n-1)}>0$, then $x_{kj}>0$ for some $j$. From $x\in \mathbb{D}^{m,n-1}$ we get that
\[ x_{1j}=x_{2j}=\cdots=x_{mj} >0. \]
This shows that $x$ is in a cell $\nu\in {\tt Ram}(\Delta_m,T_n)$ with $j\in \nu(i)$ for all $1\leq i \leq m$. One vertex of $\nu$ is
\[ \eta: \{1,2,\ldots, m\} \rightarrow \{1,2,\ldots, n\} \textrm{ defined by } \eta(i)=j. \]
\end{itemize}
The vertex $j$ is a cell of $D_m$, so $\eta^{-1}(j)=\{1,2, \ldots, m\} \in \partial \Delta_m$. But that is a contradiction, and we can conclude that
no point of ${\tt Ram}(\partial \Delta_m,D_n)$ is on the diagonal $\mathbb{D}^{m,n-1}$.
\end{proof}

This lemma was used by Volovikov \cite{V} in his proof of the topological Tverberg theorem in the prime power case. It follows from using
$\mathbb{Z}_p^{k}$--equivariant cohomology with coefficients in $\mathbb{Z}_p$, and that $\mathbb{R}^{m(n-1)} \setminus \mathbb{D}^{m,n-1}$
is homotopy equivalent to a sphere of dimension $(n-1)(m-1)-1$.
\begin{lemma}\label{lemma:volovikov}
Let $\mathbb{Z}_p^{k}$ act on the CW complex $X$ without fixed points, and assume that there is a  $\mathbb{Z}_p^{k}$--equivariant map
\[ X \rightarrow  \mathbb{R}^{m(n-1)} \setminus \mathbb{D}^{m,n-1} \]
where $m=p^k$ is a prime power. 
If  $\tilde{H}^i(X;\mathbb{Z}_p)=0$ for all $i\leq l$, then $l<(n-1)(m-1)-1$.
\end{lemma}

This is the main result of the paper.
\begin{theorem}\label{theTheorem}
Let $\Sigma$ be a simplicial complex and $p^k$ a prime power.
If  
\[ \tilde{H}^i(  {\tt Ram}( \partial \Delta_{p^k}, \Sigma)  ;\mathbb{Z}_p)=0\]
 for all $i \leq (n-1)(p^k-1)-1$ then
the vertices of $\Sigma$ cannot be $n$-colored such that each color is a face of $\Sigma$. 
\end{theorem}
\begin{proof}
Assume to the contrary that ${\tt Ram}(\Sigma, D_n)$ is non-empty.
By Proposition \ref{prop:fpf} and Lemma~\ref{lemma:intoDiag} there are fixed-point free
$\mathbb{Z}_p^k$--actions and equivariant maps
\[ {\tt Ram}( \partial \Delta_{p^k},\Sigma) \rightarrow {\tt Ram}( \partial \Delta_{p^k}, D_n)  \rightarrow  \mathbb{R}^{p^k(n-1)} \setminus \mathbb{D}^{p^k,n-1} \]
But then it contradicts Lemma~\ref{lemma:volovikov} that $\tilde{H}^i(  {\tt Ram}( \partial \Delta_{p^k}, \Sigma)  ;\mathbb{Z}_p)=0$ for all $i \leq (n-1)(p^k-1)-1$.
\end{proof}

\begin{corollary}[Lov\'asz \cite{L}]
If ${\tt Hom}(K_2,G)$ is $(n-2)$--connected then the chromatic number of the graph $G$ is larger than $n$.
\end{corollary}
\begin{proof}
Set $p^k=2$ and $\Sigma={\tt Ind}(G)$.
\end{proof}

In the next section we prove the corollary relevant to Ramsey theory.

\section{Ramsey Theory}

For any finite graph $G$ let the \emph{Ramsey number} $R(G;n)$ be the smallest number $N$ such that any edge coloring of $K_N$ with
$n$ colors contains a one-colored copy of $G$. It is a theorem by Ramsey that his numbers exists. Even for the diagonal case $R(K_m;2)$
it is very hard to give an upper bound \cite{C}. 

\begin{definition}\label{def:edgeComplex}
 Let $G$ be a graph and $N$ a positive integer.
The simplicial complex $\Sigma(G;N)$ has vertex set $E(K_N)$, and $\sigma\in \Sigma(G;N)$ if $G$ is not a subgraph of $K_N[\sigma]$.
\end{definition}

\begin{corollary}
Let $G$ be a graph, $N$ a positive integer, and $p^k$ a prime power.
If  
\[ \tilde{H}^i(  {\tt Ram}( \partial \Delta_{p^k}, \Sigma(G;N))  ;\mathbb{Z}_p)=0\]
 for all $i \leq (n-1)(p^k-1)-1$ then
 $R(G;n) \leq N.$
\end{corollary}
\begin{proof}
Insert Definition~\ref{def:edgeComplex} into Theorem~\ref{theTheorem}.
\end{proof}

Here is an explicit version:

\begin{definition}\label{def:rain}
The \emph{Rainbow complex to prove that $R(G;n)\leq N$ using $m$,}
is a polyhedral subcomplex of the product
\[ \prod_{i=1}^m \Delta_{E(K_N)}. \]
The vertices of the simplex $\Delta_{E(K_N)}$ are indexed by the edges of $K_N$. Any vertex of the product
corresponds to a function
\[ \eta : \{1,2,\ldots, m\} \rightarrow E(K_N). \]
The Rainbow complex is the induced subcomplex on the vertices $\eta$ such that $G$ is a subgraph of 
$\eta( \{1,2, \ldots,  m\} ).$
\end{definition}

\begin{corollary}\label{cor:rain}
Let $X$ be the Rainbow complex to prove that $R(G;n)\leq N$ using $m$. If $m=p^k$ is a prime power
and $\tilde{H}^i( X   ;\mathbb{Z}_p)=0$ for all $i \leq (n-1)(p^k-1)-1$, then  $R(G;n)\leq N$.
\end{corollary}

\end{document}